\documentclass[letterpaper,12pt]{amsart}
\usepackage{latexsym,array,delarray,amsthm,amssymb,epsfig}

\theoremstyle{plain}
\newtheorem{thm}{Theorem}[section]
\newtheorem{lemma}[thm]{Lemma}
\newtheorem{prop}[thm]{Proposition}
\newtheorem{cor}[thm]{Corollary}
\newtheorem{conj}[thm]{Conjecture}
\newtheorem*{thm*}{Theorem \ref{thm:main}}
\newtheorem*{lemma*}{Lemma}
\newtheorem*{prop*}{Proposition}
\newtheorem*{cor*}{Corollary}
\newtheorem*{conj*}{Conjecture}

\theoremstyle{definition}
\newtheorem{defn}[thm]{Definition}
\newtheorem{ex}[thm]{Example}
\newtheorem{pr}[thm]{Problem}
\newtheorem{alg}[thm]{Algorithm}
\newtheorem{ques}[thm]{Question}

\theoremstyle{remark}
\newtheorem*{rmk}{Remark}

\newcommand{\zz}{\mathbb{Z}}

\newcommand{\qq}{\mathbb{Q}}
\newcommand{\rr}{\mathbb{R}}

\newcommand{\kk}{\mathbb{K}}

\newcommand{\calc}{\mathcal{C}}

\newcommand{\ind}{\mbox{$\perp \kern-5.5pt \perp$}}

\begin{document}

\title{Identifiable reparametrizations of linear compartment models}
\author{Nicolette Meshkat}
\author{Seth Sullivant}

             \email{ncmeshka@ncsu.edu}
              \email{smsulli2@ncsu.edu } 
               
              \address{Department of Mathematics, Box 8205, North Carolina State University, Raleigh, NC, 27695-8205, USA }     
           

\maketitle

\begin{abstract}
Structural identifiability concerns finding which unknown parameters of a model can be 
quantified from given input-output data.  Many linear ODE models,  
used in systems biology and pharmacokinetics, are unidentifiable, which 
means that parameters can take on an infinite number of values and yet yield 
the same input-output data.  We use commutative algebra and graph theory to
 study a particular class of 
unidentifiable models and find conditions to obtain
 identifiable scaling reparametrizations of these models.  
Our main result is that the existence of an identifiable scaling
reparametrization is equivalent to the existence of a scaling
reparametrization by monomial functions.  We provide an algorithm for finding these reparametrizations when they exist and
partial
results beginning to classify graphs which possess an identifiable
scaling reparametrization.
\par
\noindent \emph{Keywords:} Identifiability, Compartment models,  Reparametrization
\end{abstract}


\section{Introduction}

Parameter identifiability analysis for dynamic system ODE models addresses the question of which unknown parameters can be quantified from given input-output data.  This paper is concerned with structural identifiability analysis, that is whether the parameters of a model could be identified if perfect input-output data (noise-free and of any duration required) were available.  If the parameters of a model have a unique or finite number of values given input-output data, then the model and its parameters are said to be \textit{identifiable}.  However, if some subset of the parameters can take on an infinite number of values and yet yield the same input-output data, then the model and this subset of parameters are called \textit{unidentifiable}.  In such cases, we attempt to reparametrize the model to render it identifiable.  

There have been several methods proposed to find these identifiable reparametrizations.  Evans and Chappell \cite{Evans} use a Taylor Series approach, Chappell and Gunn \cite{Chappell} use a similarity transformation approach, and both Ben-Zvi et al \cite{Ben-Zvi} and Meshkat et al \cite{Meshkat} use a differential algebra approach to find identifiable reparametrizations of nonlinear ODE models (see \cite{Miao2011} for a survey of methods).
However, as demonstrated in \cite{Evans}, there is no guarantee that these reparametrizations will be rational.  For practical applications, e.g. in systems biology, a rational reparametrization is desirable.  The motivation for this paper is to address the following question for linear systems:

\begin{ques}
For which linear ODE models does there exist a rational identifiable
reparametrization?
\end{ques}

In this paper, we focus on \emph{scaling reparametrizations}, which are reparametrizations
that are obtained by replacing an unobserved variable by a scaled
version of itself, and updating the model coefficients accordingly.  We will answer the above question and provide an algorithm (see Algorithm \ref{alg:mainalg}) which takes as its input a system of linear ODEs  with parametric coefficients and gives as its output an identifiable scaling reparametrization, if it exists, or shows that no identifiable scaling reparametrization
exists.

Our main result gives a precise characterization of when
a scaling reparametrization exists, for a specific family of
linear ODE models.

\begin{thm}\label{thm:main}
Consider the linear compartment model with associated strongly connected graph
$G$, where the input and output are in the same compartment.  The following conditions are equivalent for this model:
\begin{enumerate}
\item  The model has an identifiable scaling reparametrization.
\item  The model has an identifiable scaling reparametrization by
monomial functions of the original parameters.
\item  The dimension of the image of the double characteristic polynomial
map associated to $G$ is equal to the number of linearly independent
cycles in $G$.
\end{enumerate}
\end{thm}

Note the two key features of the theorem: by part (2) we only need to consider
monomial scaling reparametrizations of the model, and by part
(3) checking for the existence of an identifiable monomial rescaling is
equivalent to determining the dimension of the image of a
certain algebraic map, the \emph{double characteristic polynomial map}.
Theorem   \ref{thm:main} leaves open the problem of characterizing
the graphs $G$ which satisfy the necessary dimension requirements,
but we provide a number of partial results, including
upper bounds on the number of edges that can appear, and constructions
of families of graphs which realize the dimension bound, and
hence have identifiable reparametrizations by monomial rescalings.

The organization of the paper is as follows. The next section provides  
introductory material on compartment models, how to derive the
input-output equation, identifiability, and reparametrizations.
Section \ref{sec:ident} also introduces the main algebraic object of study
in this paper: the double characteristic polynomial map.
Section \ref{sec:cycles} explains how the identifiability
problem relates to the directed cycles in the graph $G$,
and how the cycle structure gives bounds on the dimension of
the image of the double characteristic polynomial map.
Section \ref{sec:mono} contains a proof of Theorem 
\ref{thm:main}, which reduces the problem of characterizing
the graphs which have a scaling reparametrization to the
problem of calculating the dimension of the image of the double 
characteristic polynomial map.
Section \ref{sec:dim} 
includes various combinatorial
constructions to achieve the correct dimension, as well as some necessary
conditions.  In particular, we show that all minimal inductively strongly connected
graphs achieve the correct dimension, and hence have an identifiable
scaling reparametrization.  Section \ref{sec:comp} summarizes our theoretical
results with an algorithm for computing an identifiable scaling 
reparametrization (if one exists).  Section \ref{sec:comp} also includes
 the results of systematic
computations for graphs with few vertices, and contains conjectures
based on the results of those computations.


\section{Identifiability and Reparametrizations}\label{sec:ident}

Let $G$ be a directed graph with $m$ edges and $n$ vertices.  We associate a matrix $A(G)$ to the graph in the following way:

\[
  A(G)_{ij} = \left\{ 
  \begin{array}{l l l}
    a_{ii} & \quad \text{if $i=j$}\\
    a_{ij} & \quad \text{if $j\rightarrow{i}$ is an edge of $G$}\\
    0 & \quad \text{otherwise,}\\
  \end{array} \right.
\]
where each $a_{ij}$ is an independent real parameter.  For brevity, we will use $A$ to denote $A(G)$.

Consider the ODE system of the form,
\begin{equation} \label{eq:main}
\dot{x}(t)=Ax(t)+u(t) \ \ \ \ \ \ \ \ y=x_1
\end{equation}
where $x(t) \in\rr^n$ and $u(t) \in\rr^n$, with $u(t)={\begin{pmatrix}
u_1(t) & 0 & \ldots & 0
\end{pmatrix}}^T$.

Such models are called  \textit{linear compartment models} \cite{Chapman}, where $x$ is the state variable, $u$ is the input vector, $y$ is the output, and the nonzero entries $a_{ij}$ of $A$ are independent parameters.  Since $G$ is a directed graph with $m$ edges and $n$ vertices, the dimension of the parameter space of this model is $m + n$.  Note that $u$ has only one nonzero entry in the first coordinate, and that our
output is $y = x_{1}$, which is also from the first compartment.  Hence, in this
paper we only consider models where there is a single  input and output and both
 are in the same compartment.
Note that we can only observe the input $u_{1}$ and the output $y$:  the state variable $x$ and the parameter entries of $A$ are unknown.

\begin{figure}
\begin{center}
\resizebox{!}{3.5cm}{
\includegraphics{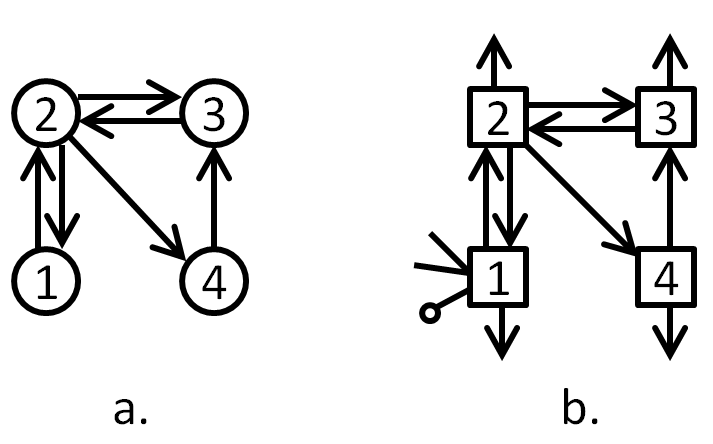}}
\end{center}\caption{a.) A graph with four vertices and b.) A compartment model}
\end{figure}

\begin{ex} \label{ex:mainex} 
For the directed graph $G$ on four vertices with six edges in Figure 1a, the ODE system has the following form:

$$
\begin{pmatrix} 
\dot{x}_1 \\
\dot{x}_2 \\
\dot{x}_3 \\ 
\dot{x}_4 \end{pmatrix} = {\begin{pmatrix} 
a_{11} & a_{12} & 0 & 0 \\
a_{21} & a_{22} & a_{23} & 0 \\
0 & a_{32} & a_{33} & a_{34} \\
0 & a_{42} & 0 & a_{44} 
\end{pmatrix}} {\begin{pmatrix}
x_1 \\
x_2 \\
x_3 \\
x_4 \end{pmatrix} } + {\begin{pmatrix}
u_1 \\
0 \\
0 \\
0 \end{pmatrix}}
$$
$$ y=x_1.$$

\end{ex}

From a biological perspective, we think of the vertices as compartments and the unknown parameters as exchange rates between the compartments.  
The off-diagonal entries $a_{ij}$ of $A$ are the instantaneous rates of transfer of material
from the $j$th compartment to the $i$th compartment.  If there is 
no edge $j \to i$, then then is no direct transfer of material from compartment
$j$ to compartment $i$.  In addition,
each compartment is assumed to have a leak, i.e.~an outflow of material
from that compartment outside the system.
In a typical biological setup, the diagonal entries $a_{ii}$ are expressed as the negative sum of the leak, written as $a_{0i}$, and the other entries in the $i$th column, so that $a_{ii}=-a_{0i}-\sum_{j\neq{i}}{a_{ji}}$. 
Hence, for biological applications, we would assume that our matrix $A(G)$ has
nonnegative off-diagonal entries, negative diagonal entries, and with the
leak assumption, it will be strictly diagonally dominant.  

In a
typical setup from a biological application, the graph from
Example \ref{ex:mainex} would have the compartment model representation
in Figure 1b.  The square vertices represent compartments, outgoing arrows
from each compartment represent leaks, the edge with a circle coming out
of compartment $1$ represents the output, 
and the arrowhead pointing into compartment $1$ represents the input.
However, since we will have a leak at every compartment, and always
have the input and output in the same compartments, we will not draw
these features in our graphs throughout the paper.

Since we can only observe the input and output to the system, we are interested in relating these quantities by forming an \textit{input-output equation}, i.e.~an equation purely in terms of input, output, and parameters.  We will use the input-output equation to address the problem of identifiability of the model parameters, although there are other methods to do so, as demonstrated in \cite{Miao2011}.  There have been several methods proposed to find the input-output equations of nonlinear ODE models \cite{Ljung, Meshkat2}, but for linear models the problem is much simpler.

\begin{thm}\label{thm:ioeqn}
Let $A_{1}$ be the submatrix of $A$ obtained by deleting the first row and column
of $A$.     
Let $\tilde{f}$ be the characteristic polynomial of $A$, $\tilde{f}_{1}$ the
characteristic polynomial of $A_{1}$,  $g = \gcd(\tilde{f}, \tilde{f}_{1})$, 
$f = \tilde{f}/g$ and $f_{1} = \tilde{f}_{1}/g$.
Then the input-output equation of the system  (\ref{eq:main}) is
$$
f(\textstyle\frac{d}{dt})y  =  f_{1}(\textstyle\frac{d}{dt}) u_{1}.
$$

In particular, if the characteristic polynomials of $A$ and $A_{1}$ are relatively
prime then, the input-output equation of the system (\ref{eq:main}) is 
\begin{equation} \label{eq:ioeqn}
y^{(n)}+c_1y^{(n-1)}+\cdots+c_ny = u_{1}^{(n-1)}+d_{1}u_{1}^{(n-2)}+\cdots+d_{n-1}u_{1}
\end{equation}
where  $c_1,\ldots,c_n$ are the  coefficients of the characteristic polynomial of $A$
 and  $d_{1},\ldots,d_{n-1}$ are the  coefficients of the characteristic polynomial of $A_{1}$. 
 \end{thm}

\begin{proof} 

We can re-write our ODE system as:
$$(\partial{I}-A)x=u$$
where $\partial$ is the differential operator $d/dt$.  Using formal manipulations with this operator, we can use Cramer's Rule to get that 
$$x_1=det(A_2)/det(\partial{I}-A)$$
where $A_2$ is the matrix $(\partial{I}-A)$ with the first column replaced by $u$.  Since $u$ has $u_1$ as its first entry and zeros otherwise, then $det(A_2)$ can be simplified as $det(\partial{I}-A_1)u_1$, where $A_1$ is the submatrix of $A$ where the first row and first column have been deleted.  Then replacing $x_1$ with $y$, we get the input-output equation:
$$det(\partial{I}-A)y=det(\partial{I}-A_1)u_1.$$

In other words, $\tilde{f}(\frac{d}{dt})y  =  \tilde{f}_{1}(\frac{d}{dt}) u_{1}$.  Dividing both sides by $g = \gcd(\tilde{f}, \tilde{f}_{1})$, we get $f(\frac{d}{dt})y  =  f_{1}(\frac{d}{dt}) u_{1}$.  

If $g=1$, we have that the input-output equation is of the form:
$$
y^{(n)}+c_1y^{(n-1)}+\ldots +c_ny = u_{1}^{(n-1)}+d_{1}u_{1}^{(n-2)}
+ \ldots +d_{n-1}u_{1}
$$
where the coefficients $c_1,\ldots,c_n$ are the $n$ coefficients of the characteristic polynomial of $A$ and the coefficients $d_{1},\ldots,d_{n-1}$ are the $n-1$ coefficients of the characteristic polynomial of $A_{1}$. 
\end{proof}

\begin{rmk}
We will show in Section \ref{sec:cycles} that the input-output 
equation has the form (\ref{eq:ioeqn})
for generic choices of the parameters $A$
if and only if $G$ is strongly connected.
\end{rmk}

\begin{ex} \label{ex:mainio}
For the graph in Example \ref{ex:mainex}, the input-output equation is:

\begin{align*}
y^{(4)}-E_{1}(a_{11},a_{22},a_{33},a_{44})y^{(3)}  
+ ( E_{2}(a_{11},a_{22},a_{33},a_{44})-a_{12}a_{21}-a_{23}a_{32})y^{(2)} \\
-(E_{3}(a_{11},a_{22},a_{33},a_{44})
-a_{11}a_{23}a_{32}-a_{12}a_{21}a_{33} + a_{23}a_{34}a_{42}-a_{12}a_{21}a_{44}
-a_{23}a_{32}a_{44})y^{'}\\
+(E_{4}(a_{11},a_{22},a_{33},a_{44}) + a_{11}a_{23}a_{34}a_{42}-a_{11}a_{23}a_{32}a_{44}-a_{12}a_{21}a_{33}a_{44})y \\ =  \quad \quad u_{1}^{(3)}-E_{1}(a_{22},a_{33},a_{44})u_{1}^{(2)}
+(E_{2}(a_{22},a_{33},a_{44})-a_{23}a_{32})u_{1}^{'} \\
-(E_{3}(a_{22},a_{33},a_{44})+a_{23}a_{34}a_{42}-a_{23}a_{32}a_{44})u_{1}
\end{align*}
where $E_{k}(z_{1}, \ldots, z_{m})$ denotes the $k$-th elementary symmetric polynomial
in $z_{1}, \ldots, z_{m}$.
\end{ex}

\smallskip

Identifiability of an input-output equation concerns whether
it is possible to recover the parameters of the model (in our
case, the entries of $A$) only observing the relations among the
input and output variables.  In other words, we assume that
we observe specific values of the coefficients $c_{1}, \ldots, c_{n}$
and $d_{1}, \ldots, d_{n-1}$ and we ask whether it is possible
to recover the entries of $A$.   More generally, we can
ask for functions of the parameters $A$ which can be computed 
from $c_{1}, \ldots, c_{n}$
and $d_{1}, \ldots, d_{n-1}$.  Such a function is called 
an identifiable function.  We make these notions precise in 
generality.

\begin{defn}
Let $c$ be a function $c:\Theta\rightarrow{\kk^{m_{2}}}$, where
$\Theta \subseteq \kk^{m_{1}}$ and $\kk$ is a field.
The model parameters in $\Theta$ are \emph{globally identifiable} from $c$ if and only if the map $c$ is injective.  A subset of the model parameters in $\Theta$ are
\emph{locally identifiable} from $c$ if and only if the map $c$ is finite-to-one.  
A subset of the model parameters in $\Theta$ are
\emph{unidentifiable} from $c$ if and only if the map $c$ is 
infinite-to-one.    
\end{defn}

It is often the case that parameters might fail to be identifiable,
but only on a small subset of parameter space.  In this case,
we can say that identifiability holds generically.  For example,
the model parameters in $\Theta$ are \emph{generically globally
identifiable} from $c$ if there is a dense open subset $U$ of
$\Theta$, such that $c : U \rightarrow \kk^{m_{2}}$ is globally identifiable.
Similarly, we can define generically locally identifiable, and
generically unidentifiable.

\begin{rmk}
For brevity, we make the convention for the remainder of the paper
that identifiable means generically locally identifiable and
unidentifiable means generically unidentifiable.
\end{rmk}

We can also speak of identifiability of individual functions.

\begin{defn}
Let $c$ be a function $c:\Theta\rightarrow{\kk^{m_{2}}}$, where
$\Theta \subseteq \kk^{m_{1}}$ and $\kk$ is a field.  A function
$f : \Theta \rightarrow \kk$ is \emph{globally identifiable} from
$c$ if there exists a function $\Phi: \kk^{m_{2}} \rightarrow \kk$
such that $\Phi \circ c  = f$.  The function $f$ is
locally identifiable if there is a finitely multivalued function
$\Phi: \kk^{m_{2}} \rightarrow \kk$
such that $\Phi \circ c  = f$.
\end{defn}

Similarly, we can define generic identifiability of a function.
For brevity, for the remainder of the paper, when we say that
a function is identifiable, we will mean that it is generically
locally identifiable.

\begin{prop}
Suppose that $\Theta$ is a full dimensional subset of $\kk^{m_{1}}$
and $c$ is a rational map.  Then the model is identifiable
from $c$ if and only if the dimension of the image of $c$
is $m_{1}$. 
\end{prop}

Also:

\begin{prop}
Suppose that $\Theta$ is a full dimensional subset of $\kk^{m_{1}}$
and $c$ is a rational map, and $f: \Theta \rightarrow \kk$
is a function.  Then $f$ is identifiable if and only if
$$\kk(f, c_{1}, \ldots, c_{m_{1}})/ \kk( c_{1}, \ldots, c_{m_{1}})$$
is a finite degree field extension.
\end{prop}

Identifiability of the map $c$ or a function $f$ can
be tested in specific instances using Gr\"obner basis calculations.  See e.g.~\cite{Garcia2010,Meshkat}.

In the setting of linear compartment models, we have a graph
$G$ with $n$ vertices and $m$ directed edges.
The parameter space
$\Theta \subseteq \rr^{m +n}$ consists of those matrices
whose zero pattern is induced by the graph $G$, positive off-diagonal
entries, negative diagonal entries, and strictly diagonally dominant.
The map $c: \Theta \rightarrow \rr^{2n-1}$ is the map
that takes a matrix $A \in \Theta$ to the vector 
$$(c_{1}(A), \ldots,
c_{n}(A), d_{1}(A), \ldots, d_{n-1}(A))$$
of characteristic polynomial coefficients.
The map $c: \Theta \rightarrow \rr^{2n-1}$ is called the 
\emph{double characteristic polynomial map}.

\begin{ex}\label{ex:mainident} 
For the graph in Example \ref{ex:mainex}, 
the image of the double characteristic polynomial map has
dimension seven.  A set of seven algebraically independent
identifiable functions is
$\left\{a_{11}, \: a_{22}, \: a_{33}, \: a_{44}, \: a_{12}a_{21}, \: a_{23}a_{32}, \: a_{23}a_{34}a_{42}\right\}$.  
For example, $a_{12}a_{21}$ is identifiable since, for this
graph
$$
a_{12}a_{21}  =  d_{2} - c_{2} + c_{1}d_{1} - d_{1}^{2}.
$$
It is easy to see
that these functions are algebraically independent (each involves
a new indeterminate).
The fact that they are identifiable follows from the material in
Sections \ref{sec:cycles}, \ref{sec:mono}, and \ref{sec:dim}.
\end{ex}

The linear compartment models that we focus on in the present paper (that is,
where the diagonal of $A$ contains algebraically independent parameters,
or equivalently, every compartment has a leak)
are never identifiable, except in the
trivial case of a graph with one vertex.  This will
be explained in detail in the subsequent sections.  This,
however, forces us to look for identifiable 
reparametrizations of our model.

\begin{defn}
A reparametrization of the input-output equation of a model is a 
map $q: \rr^{m_{3}} \rightarrow \rr^{m_{1}}$
such that the image of $c \circ q$ equals the image of $c$.
The reparametrization is identifiable if the composed map $c \circ q$
is identifiable.
\end{defn}

Since the new parametrization is $c \circ q$, there must exist a map $\Phi: {\rm im}\, c  \rightarrow \rr^{m_{3}}$ which is the (local)
inverse of $c \circ q$.  Since $q$ must be locally injective, this
implies that the map $\Phi$ consists of identifiable functions of the map $c$.
This argument is also reversible (e.g.~by the implicit function theorem).
Hence, finding an identifiable reparametrization is, from a theoretical
standpoint, equivalent to finding $d$ algebraically
independent functions $f_{1}, \ldots, f_{d}$ that are identifiable from $c$, where $d=\dim {\rm im} \, c$.
Once these identifiable functions $f$ are found, the goal is to determine a reparametrization of our original model that yields an input-output equation with coefficients $c \circ q$.  In summary:

\begin{prop}\label{prop:rewrite}
A reparametrization $q: \rr^{m_{3}} \rightarrow \rr^{m_{1}}$ of 
$c : \rr^{m_{1}} \rightarrow \rr^{m_{2}}$ is identifiable 
if and only if there is a map $\Phi: {\rm im}\, c  \rightarrow \rr^{m_{3}}$
such that $\Phi \circ c = f$ consists of identifiable functions
from $c$.
\end{prop}

A more subtle, and not quite mathematical, issue is that
we want our reparametrization to both involve only relatively
simple functions, and have an intuitively simple connection to our
original model, without dramatic shifts in the parametrization.  
A common way to find
such a  reparametrization is via a rational scaling of the state variables 
\cite{Chappell, Evans, Meshkat}, while other methods, e.g. affine maps of the state variables, can also be employed \cite{Denis-Vidal}.  A scaling reparametrization is preferred over a more complicated type of reparametrization since it respects the biological properties of the original model.  As we will see,
when identifiable rescalings exist, they can always be made rational.
For example, we would like to find a scaling:
$$X_i=f_i(A)x_i$$
such that the reparametrized model is identifiable, i.e.~purely in terms of a fewer number of identifiable functions of parameters.  Since $y=x_1$ is observed, we require that $f_1(A)=1$.  In this way, the input and output variables remain intact.
Scaling reparametrizations have the effect of nondimensionalizing the
quantities that are being rescaled.  That is, from input-output data
we would not be able to estimate the values of those unobserved variables,
but we can predict how their relative size changes as we change parameters.   

The rescaling induced by the functions $f_{1}, \ldots, f_{n}$ maps the matrix   $A$ to $DAD^{-1}$, where $D=Diag(f_i(A))$.  In other words, the entries of $A$ become:
$$a_{ij}\mapsto{a_{ij}f_i(A)/f_j(A)}$$
For this reparametrization to be identifiable, this means that the new coefficients of the state variables,
$$a_{ij}f_i(A)/f_j(A)$$
are themselves functions of identifiable functions of parameters by Proposition \ref{prop:rewrite}.

\begin{ex} \label{ex:mainreparam} 
From the graph in Example \ref{ex:mainex}, a possible rescaling is $X_1=x_1,X_2=a_{12}x_2, X_3=a_{12}a_{23}x_3,X_4=a_{12}a_{23}a_{34}x_4$.
This yields the reparametrized system:
$$
\begin{pmatrix} 
\dot{X}_1 \\
\dot{X}_2 \\
\dot{X}_3 \\ 
\dot{X}_4 \end{pmatrix} = {\begin{pmatrix} 
a_{11} & 1 & 0 & 0 \\
a_{12}a_{21} & a_{22} & 1 & 0 \\
0 & a_{23}a_{32} & a_{33} & 1 \\
0 & a_{23}a_{34}a_{42} & 0 & a_{44} 
\end{pmatrix}} {\begin{pmatrix}
X_1 \\
X_2 \\
X_3 \\
X_4 \end{pmatrix} } + {\begin{pmatrix}
u_1 \\
0 \\
0 \\
0 \end{pmatrix}}
$$
$$ y=X_1,$$
which is identifiable since each coefficient is a function of
the seven identifiable functions from Example \ref{ex:mainident}.
We will see in Section \ref{sec:mono} exactly how this reparametrization can be found.

\end{ex}

Let $r :  \rr^{n +m}  \rightarrow \rr^{n+m}$ be the rescaling map
associated to the functions $f_{i}$:
$$
r_{ij}(A)  =  a_{ij}f_i(A)/f_j(A).
$$

\begin{prop}  \label{prop:dimdrop} The dimension of the image of the rescaling map $r$ is greater than or equal to $ \dim \Theta - (n-1)$.
\end{prop}

\begin{proof} 
We can calculate this by finding the Jacobian.  To simplify the calculation we first take the logarithm and call this function $\varphi_{ij}$:
$$\varphi_{ij}=log(f_i(A))-log(f_j(A))+log(a_{ij})$$
Then the derivative is:
$$
\frac{\partial{\varphi_{ij}}}{\partial{a_{kl}}}=
\frac{1}{f_i(A)}\frac{\partial{f_i(A)}}{\partial{a_{kl}}}-
\frac{1}{f_j(A)}\frac{\partial{f_j(A)}}{\partial{a_{kl}}}+
\delta_{kl,ij}\frac{1}{a_{ij}}
$$
Thus, the Jacobian $J(\varphi)$ can be written as:
$$
Diag(\frac{1}{a_{ij}})+J(a_{ij}\mapsto{log(f_i(A))})\cdot E(G)
$$
where $J(a_{ij}\mapsto{log(f_i(A))})$ is the Jacobian of the mapping $a_{ij}\mapsto{log(f_i(A))}$ and $E(G)$ is the $n$ by $m$ incidence matrix (defined in Section \ref{sec:cycles}, see Eq (\ref{eq:eg})).  We can write this in shorthand notation as:
$$
J(\varphi) = D + J \cdot E(G)
$$
Then we have that $rank(J(\varphi))\geq{rank(D)-rank(J \cdot E(G))}$.  
From Proposition \ref{prop:dimeg}, we have that ${\rm rank}(E(G)) \leq n-1 $, which implies ${\rm rank}(J\cdot E(G))\leq n-1 $, and thus:
$$
{\rm rank}(J(\varphi))\geq m+n - (n-1) = m+1
$$
In other words, the dimension drops by at most $n-1$.  
\end{proof}

This theorem gives us the maximal number of parameters of a model with an
identifiable scaling reparametrization.

\begin{cor} \label{cor:maxedges} 
Let $G$ be a graph for which an identifiable scaling reparametrization of
the system (\ref{eq:main}) exists.  Then $G$ has at most $2n -2$ edges.
\end{cor}

\begin{proof}  Proposition \ref{prop:dimdrop} gives us the minimal dimension of the image of $r$, which is $\dim \Theta - (n-1) = m+1$.  
The dimension of the image of $c$ is at most $2n-1$.  Since an identifiable reparametrization means that the dimension of the image of $c \circ q$ is the dimension of the image of $c$, the dimension of the image of $r$ must be at most $2n-1$. Thus $m$ is at most $2n-2$.
\end{proof}

This leads us to the main problem to be studied in the remainder 
of the paper:

\begin{pr}\label{pr:reparam}
For which graphs $G$ with $n$ vertices and $\leq 2n-2$ edges
does there exist a generically locally identifiable scaling 
reparameterization
of the system (\ref{eq:main}) associated to the graph $G$?
\end{pr}

Since local identifiability is 
completely determined by dimension, we can 
break the problem into three parts:

\begin{enumerate}
\item  Determine the dimension $d$ of the image of the double characteristic polynomial map $c$ as a function of the graph $G$.
\item  Find a set of $d$ algebraically independent identifiable functions from $c$.
\item  Find an identifiable reparametrization of the ODE system, using the
$d$ algebraically independent functions. 
\end{enumerate}

It is these three problems which we address in the subsequent sections.


\section{Cycles and Monomials}\label{sec:cycles}

In this section, we begin to relate the study of Problem \ref{pr:reparam}
to the particular structure of the graph $G$.  The cycles in 
$G$ play a crucial role, because of their appearance in the 
calculation of the characteristic polynomial.  This section
describes this relationship and relates the structure of
cycles in the graph to the problem of finding identifiable
reparametrizations.  The connection between
the cycle structure in the graph $G$ and identifiability
of the associated linear compartment model has been employed 
in other works (see \cite{Audoly, Godfrey}).  However, our paper
appears to be the first to use this structure to  study the
existence of identifiable reparametrizations.

\begin{defn} A \emph{closed path} in a directed graph $G$ is a sequence of 
vertices $i_{0},i_{1}, i_{2}, \ldots, i_{k}$ with $i_{k} = i_{0}$ and
such that $i_{j} \to i_{j+1}$ is an edge for all $j = 0, \ldots, k-1$.
A \emph{cycle} in $G$ is a closed path with no repeated vertices.
To a cycle $C = i_{0},i_{1}, i_{2}, \ldots, i_{k}$, we associate the
monomial $a^{C} = a_{i_{1}i_{2}}a_{i_{2}i_{3}}\cdots a_{i_{k}i_{1}}$,
which we refer to as a \emph{monomial cycle}.
\end{defn}

Note that the diagonals of $A$ are monomial $1$-cycles from the graph $G$.

\begin{thm} \label{thm:funccycles} 
The coefficients $c_i$ of the characteristic polynomial $\lambda^n+\sum_{i=1}^{n}{c_{i}\lambda^{n-i}}$ of $A$ are polynomial functions in terms of the monomial cycles, $a^{C} = a_{i_{1}i_{2}}a_{i_{2}i_{3}}\cdots a_{i_{k}i_{1}}$, of the graph $G$.  

Specifically, let $\calc(G)$ be the set of all cycles in $G$.
Then 
$$c_{i} =  (-1)^i\sum_{C_{1}, \ldots, C_{k} \in \calc(G)} \prod_{j = 1}^{k}  {\rm sign}(C_{j}) a^{C_{j}},$$
where the sum is over all collections of vertex disjoint cycles
involving exactly $i$ edges of $G$, and ${\rm sign}(C) = 1$ if
$C$ is odd length and ${\rm sign}(C) = -1$ if $C$ is even length.
\end{thm}

This follows from the expansion of the determinant, and breaking each
permutation into its disjoint cycle decomposition.

\begin{ex}
Looking at the input-output equation for the graph from Example \ref{ex:mainex},
which appears in Example \ref{ex:mainio}, we see that the coefficient of
$y'$ is
$$-(E_{3}(a_{11},a_{22},a_{33},a_{44})
-a_{11}a_{23}a_{32}-a_{12}a_{21}a_{33} -a_{12}a_{21}a_{44}
-a_{23}a_{32}a_{44} + a_{23}a_{34}a_{42}).$$
The elementary symmetric function gives all terms that 
come from products of three $1$-cycles.  All the terms with a minus sign
come from products of a $2$-cycle and a $1$ cycle, and the final term comes
from the single $3$-cycle in the graph.
\end{ex}

\begin{defn}
A graph $G$ is \emph{strongly connected} if there is a directed path from
any vertex to any other vertex.  Equivalently, $G$ is strongly connected
if the graph is connected and every edge belongs to a cycle.
\end{defn}

\begin{prop}\label{prop:strong}
Let $G$ be a graph, $A = A(G)$ be the associated matrix of indeterminates, and $A_{1}$ be the submatrix of $A$ obtained by deleting the first row and column
of $A$.  
Let
$f$ and $f_{1}$ be the characteristic polynomials of $A$ and $A_{1}$
respectively.  Then $f$ and $f_{1}$ have a common factor if and only if
$G$ is not strongly connected.
\end{prop}

\begin{proof}
If $G$ is not strongly connected after rearranging rows and columns
of $A$, it will be a block upper triangular matrix.  The
characteristic polynomial of $A$ factors as the product of the
characteristic polynomials of the diagonal blocks.  The
characteristic polynomial of $A_{1}$ will contain as factors,
all of the factors for diagonal blocks of $A$ that do not involve
row/column $1$.

On the other hand, if $G$ is strongly connected, the characteristic
polynomial of $A$ is irreducible in the polynomial ring
$\kk(A)[\lambda]$, where $\kk(A)$ is the fraction field in the entries of
$A$.  This can be seen by looking at the constant term of the
characteristic polynomial, i.e.~$\det(A)$, which itself is
irreducible in the polynomial ring $\kk[A]$.  Indeed, if $\det(A)$
was reducible, we could partition in the vertices of $G$ into
two disjoint sets such that there were no cycles passing between
those sets of vertices.  This contradicts the fact that $G$ is strongly
connected.
\end{proof}

\begin{rmk}
If $G$ is a general graph with generic parameters, then
the input-output equation will result from taking the
largest strongly connected subgraph of $G$ that contains the
vertex $1$.  With this in mind, we will focus in the
remainder of the paper only on strongly connected graphs.  
\end{rmk}

Let $\calc = \calc(G)$ be the set of all directed cycles in 
the graph $G$.  To each cycle $C = (i_{0}, \ldots, i_{k}) \in \calc$ we associate the monomial cycle
$a^{C} := a_{i_{1}i_{2}}a_{i_{2}i_{3}}\cdots a_{i_{k}i_{1}}$.
Define the \emph{cycle map} by
$$
\pi:  \rr^{m+n}  \rightarrow  \rr^{\#\calc},  \quad  A  \mapsto ( a^{C})_{C \in \calc}.
$$
Since the coefficients of the 
characteristic polynomial of $A$ and $A_{1}$ are both
polynomials in terms of the cycles of $G$, the double
characteristic polynomial map $c$ factors through the cycle map.
That is, there is a polynomial map $\phi:  \rr^{\#\calc} \rightarrow
\rr^{2n-1}$ such that $c =  \phi \circ \pi$.  As a
consequence, we have the following proposition.

\begin{prop}\label{prop:cyclemapbound}
Let $G$ be a graph with $n$ vertices and $m$ edges.  The
dimension of the image of the double
characteristic polynomial map $c$ is less than or equal to the dimension
of the image of the cycle map $\pi$.
In particular, $\dim {\rm im} \, c $ is bounded above by
the number of algebraically independent monomial cycles in $G$.
\end{prop}

Since the cycle map is a monomial map, it is easy to use
linear algebra to calculate the dimension of its image.
This is part of the connection between lattice polytopes and
toric varieties \cite{Cox2011}, though we will not require advanced
material from that theory.  The main result for our story is the
following.

\begin{thm}\label{thm:dimcycle}
Let $G$ be a strongly connected graph with $n$ vertices and $m$
edges.  Then the dimension of the image of the cycle map $\pi$ is
$m + 1$.
\end{thm}

We will phrase the proof of Theorem \ref{thm:dimcycle} in terms
of the directed incidence matrix, a tool we will also need later
in the paper.  Let $G$  have $n$ vertices, $V=\left\{1,2,\ldots,n\right\}$, and $m$ directed edges.  We can form the $n$ by $m$ directed incidence matrix $E(G)$, where

\begin{equation}\label{eq:eg}
  E(G)_{i,(j,k)} = \left\{ 
  \begin{array}{l l l}
    1 & \quad \text{if $i=j$}\\
   -1 & \quad \text{if $i=k$}\\
    0 & \quad \text{otherwise.}\\
  \end{array} \right.
\end{equation}

In other words, $E(G)$ has column vectors $e_{jk}$ corresponding to the edges 
$j\rightarrow{k}$ with $1$ in the $jth$ row, $-1$ in the $kth$ row, and $0$ 
otherwise.  Note that a $0/1$ vector in the kernel of $E(G)$ is
the indicator vector of the disjoint union of a collection of cycles in
$G$.

The rank of the directed incidence matrix is well-known 
(e.g. \cite[Prop.~4.3]{Biggs}).

\begin{prop}\label{prop:dimeg}
Let $G$ be a graph with $n$ vertices, $m$ edges, and $l$ connected components.
Then the rank of $E(G)$ is $n - l$.  Thus, the dimension
of $\ker E(G)$ is $m - n + l$.
\end{prop}

\begin{proof}[Proof of Theorem \ref{thm:dimcycle}]
If $\phi: \rr^{k_{1}} \rightarrow \rr^{k_{2}}$ is a
monomial map, the dimension of the image of $\phi$ is equal to 
the rank of the matrix whose columns are the monomials
appearing in $\phi$.
In the case of the cycle map, we should thus make a $0/1$ matrix $B$
whose columns are the cycles in $G$, and compute the rank of
that matrix.  All the one cycles, $(i,i)$, corresponding to the
monomial $a_{ii}$ contribute one dimension to the rank of $B$,
and those one cycles do not appear in any other cycles.
Hence we can reduce to a matrix $B'$ which eliminates those $n$
columns.

Thus, we are left with the $0/1$ matrix whose columns are the indicator
vectors of all cycles in $G$.  The columns of $B'$ are all in the
kernel of $E(G)$ (see e.g. \cite[Theorem ~4.5]{Biggs}).  Since $G$ is strongly connected,
$\dim \ker E(G) = m-n+1$.  Hence, it suffices to show that
the columns of $B'$ generate the kernel of $E(G)$ when $G$ is strongly
connected.

Let $v$ be an integer vector in the kernel of $E(G)$.  Since $G$ is strongly
connected, for each negative entry of $v$, there is a cycle $C$ passing
through the corresponding edge of $G$.  Let $1_{C}$ be the corresponding
integer vector.  Then for some large integer $k$, $v + k \cdot 1_{C}$ has
decreased the number of negative entries of $v$.  Continuing
in this fashion, we can assume that $v$ has no negative entries.

A nonnegative integer vector $v$ such that $E(G)v = 0$ corresponds to a 
multigraph (with edge $i \to j$ repeated $v_{ij}$ times) which has the property
that the indegree of each vertex equals the outdegree.  In such an
Eulerian graph, we can start with any edge and walk around until closing
off a cycle.  Removing that cycle results in a smaller graph with the same
property.  This process expresses $v$ as a nonnegative integer combination
of the indicator vectors of cycles.  This completes the proof.
\end{proof}

Corollary \ref{cor:maxedges} states that if $G$ is to have an
identifiable scaling reparametrization, $G$ must have at most
 $2n -2$ edges.  Theorem \ref{thm:dimcycle} and
Proposition  \ref{prop:cyclemapbound} say that this bound
on the number of edges is at least compatible with the
existence of an identifiable scaling reparametrization.
We will address this issue in the next section.


\section{Monomial Scaling Reparametrizations}\label{sec:mono}

The goal of this section is to prove Theorem \ref{thm:main},
which we restate here for simplicity.

\begin{thm*} 
Consider the linear compartment model with associated strongly connected graph
$G$, where the input and output are in the same compartment.  
The following conditions are equivalent for this model:
\begin{enumerate}
\item  The model has an identifiable scaling reparametrization.
\item  The model has an identifiable scaling reparametrization by
monomial functions of the original parameters.
\item  The dimension of the image of the double characteristic polynomial
map associated to $G$ is equal to the number of linearly independent
cycles in $G$.
\end{enumerate}
\end{thm*}

\begin{proof}
Clearly $(2) \implies (1)$.  Also, it is not difficult to see that
$(1) \implies (3)$.  Indeed, if $G$ has $n$ vertices and $m$ edges,
the dimension of the image of the double characteristic polynomial map
is $\leq m+1$, $m+1$ being the number of linearly independent cycles in 
$G$ by Proposition \ref{prop:cyclemapbound} and Theorem \ref{thm:dimcycle}.
On the other hand, Proposition \ref{prop:dimdrop} shows that the
dimension of the image of any rescaling map is $\geq n + m - (n - 1) = m+1$.
Since an identifiable reparametrization implies that the
dimension of the image of the rescaling $r$ equals the dimension
of the image of $c$, we are done.
\end{proof}

What remains to show is that $(3) \implies (2)$, and this is
the issue that we spend the rest of this section proving.
Let $E$ be the matrix obtained from $E(G)$ by deleting the first row.
Let $M$ be an $m \times (m-n+1)$ matrix who columns consist of $m-n+1$
linearly independent cycles in the graph $G$.

\begin{lemma}\label{lem:matrixeq}
Let $G$ be a strongly connected graph and suppose that 
the dimension of the image of the
double characteristic polynomial map associated to $G$ is equal to the
number of linear independent cycles in $G$.  Then the model
has an identifiable reparametrization by monomial functions 
if there exist integer matrices $C$ and $D$ such that
$$
I + CE = MD
$$
where $I$ is an $m \times m$ identity matrix.
\end{lemma}

\begin{proof}
Assume we have an ODE system as defined in the previous sections.  We perform a monomial scaling $X_i=f_{i}(A)x_i$ for $i=1,\ldots,n$, where $f_i(A)$ is a monomial in the $m$ off-diagonal entries, a subset of $\left\{a_{12},a_{13},\ldots,a_{n,n-1}\right\},$ with exponent vector $c_i=(c_{1i},c_{2i},\ldots,c_{m,i})$. Since we do not want to reparametrize $x_1$, we let $f_1(A)=1$.  Then the entries $a_{ij}$ of matrix $A$ become $a_{ij}f_i(A)/f_j(A)$. Thus, the diagonal terms, $a_{11},a_{22},\ldots,a_{nn}$, are unchanged in our reparametrization, and we only focus on off diagonal terms.

Form the matrix of exponents of the new $m$ off-diagonal coefficients of $A$ resulting from this monomial rescaling.  This matrix can be written as $I+C\cdot{E(G)}$
where $I$ is an $m$ by $m$ identity matrix, $C$ is the $m$ by $n$ matrix whose column vectors are $c_i$, and $E(G)$ is the $n$ by $m$ incidence matrix of the graph of $A$.  Since $f_{1}(A) = 1$, the first column of $C$ is all zeros.
Hence we can delete that first column and simulataneously the first row of 
$E(G)$ to see that a scaling of the type we are interested in yields the
matrix of exponent vectors of the form
$$I + CE.$$

Now assume that the the dimension of the double characteristic polynomial
map is equal to the number of linear independent cycles.
Thus, there are $m+1$ algebraically independent identifiable monomial cycles. 
 Of the $m+1$ monomial cycles we wish to reparametrize over, exactly $n$ of them are the diagonal terms $a_{11}, a_{22},\ldots,a_{nn}$, 
while the other $m-n+1$ monomial cycles are in terms of the $m$ off-diagonal elements.  

By Proposition \ref{prop:rewrite}, finding an identifiable scaling
reparametrization amounts to finding a rescaling such that 
the rescaled monomials $a_{ij}f_{i}(A)/f_{j}(A)$
are functions of the monomial cycles, which we denote by
$q_1,q_2,\ldots ,q_{m-n+1}$.
Any monomial function of  $q_1,q_2,\ldots ,q_{m-n+1}$ has the form $q_1^{d_{1i}}q_2^{d_{2i}}\cdots q_{m-n+1}^{d_{m-n+1,i}}$ for $i=1,\ldots,m$.  Let the exponent vectors of each of the monomial cycles form the columns of the matrix $M$.  Thus the matrix of exponent vectors of all of these functions of monomial cycles in terms of the original $a_{ij}$s will be 
$M\cdot{D}$
where $M$ is the $m$ by $m-n+1$ matrix who columns are the exponent vectors of each of the monomial cycles $q_1,\ldots,q_{m-n+1}$ and $D$ is the $m-n+1$ by $m$ matrix whose columns are $d_i=(d_{1i},d_{2i},\ldots,d_{m-n+1,i})$.
To say that the scaling reparametrization yields an identifiable reparametrization
is the same as saying we can find $C$ and $D$ such that these two matrices
of exponent vectors are the same, i.e.~$I + CE = MD$.  Since we
wish for a rational reparametrization, we require both $C$ and $D$ to be integer
matrices.
\end{proof}

We will prove that there always exist integer matrices $C$ and $D$ 
such that $I + C E  = MD$ in Lemma \ref{lem:zeqb}.  To do this, we
need to record some basic facts about the matrices $E$ and $M$.

\begin{lemma}\label{lem:decompm}
Let $G$ be a strongly connected graph.
Let $E$ be obtained from $E(G)$ by deleting the first row.
Let $M$ be a matrix whose columns are a set of $m-n+1$ linearly independent
cycles in $G$.  Then
\begin{enumerate}
\item  $E$ is a totally unimodular matrix, i.e.~the determinant of any 
submatrix of $E$ is $0$ or $\pm 1$.
\item  An $(n-1) \times (n-1)$ submatrix of $E$ has rank $n-1$ if and
only if the corresponding set of $n-1$ edges of $G$ is a spanning tree
of $G$.
\item  An $(m-n+1) \times (m-n+1)$ submatrix of $M$ which corresponds to the
complement of the set of edges in a spanning tree of $G$ has determinant 
$\pm 1$. 
\end{enumerate} 
\end{lemma}

\begin{proof} Part (1) is a well-known result in the theory of
totally unimodular matrices.  See e.g.~\cite[Ch.~19]{Schrijver1986}.
Note that when $G$ is connected the only relation among the
rows of $E(G)$ is that the sum of all the rows is zero.  This means
that $E$ has rank $n-1$ for a connected graph.
Thus, part (2) follows from Proposition \ref{prop:dimeg}.

Now we prove part (3).  In \cite[Thm.~5.2]{Biggs} it is shown that a lattice basis of $\ker_{\zz} E(G)$ can
be constructed by the following procedure.  Let
$T$ be a spanning tree in $G$.  Assume that the columns
of $E$ are ordered so that the first $n-1$ columns correspond to the 
edges of $T$.  Each edge $e \in G$ that is not
in $T$ can be used to form a unique (undirected) cycle using $e$ plus edges
in $T$.  This cycle yields  a vector with $0, \pm 1$ entries that is
in the kernel of $E(G)$.  Moreover, taking all the $m-n+1$ cycles
that arise in this way and putting that as the columns of  a matrix
$N$ which has the form
$$
N = \begin{pmatrix}N'  \\  I  \end{pmatrix}
$$
where $I$ is an $(m-n+1) \times (m - n + 1)$ identity matrix.

On the other hand, the proof of Theorem \ref{thm:dimcycle} showed that
the matrix $M$ also consists of a basis for $\ker_{\zz} E(G)$.  Hence
$M = NU$ where $U$ is an $(m-n+1) \times (m-n+1)$ unimodular matrix 
(i.e.~$\det U = \pm 1$).  Writing this in block form we have
$$
M = \begin{pmatrix}
M' \\ M''  
\end{pmatrix}  =
\begin{pmatrix}N'U  \\  U  \end{pmatrix}  =  NU.
$$
Thus $\det M''  = \pm 1$.
\end{proof}

\begin{lemma}\label{lem:zeqb}
For any strongly connected graph $G$, 
there exist integer matrices $C$ and $D$ such that $I+CE=MD$. 
\end{lemma}

\begin{proof} 
We can re-write the system $I+CE=MD$ as a matrix equation
\[
I=(C  M)
\begin{pmatrix}
E\\
D
\end{pmatrix}
\]
where we replace $-C$ with $C$ for simplicity.  

Let $E$ be partitioned into $(E_1 \ E_2)$, where $E_1$ is an $n-1$ by $n-1$ matrix corresponding to the edges in a spanning tree $T$.  Let $M$ be partitioned into $(M_1 \ M_2)^T$, where $M_1$ corresponds to the spanning tree $T$.  Thus, we can further partition in the form:

\[
\begin{pmatrix}
I & 0 \\
0 & I
\end{pmatrix}
=
\begin{pmatrix}
C_1 & M_1 \\
C_2 & M_2
\end{pmatrix}
\begin{pmatrix}
E_1 & E_2 \\
D_1 & D_2
\end{pmatrix}.
\]

We claim that taking $C_{1}  = E_{1}^{-1}$, $C_{2} = 0$, $D_{1} = 0$ and
$D_{2} = M_{2}^{-1}$ provides a valid integral solution to this equation.
First, note that both $C_{1}$ and $D_{2}$ will be integral matrices, by
Lemma \ref{lem:decompm}.  To show that these choices solve the matrix
equation, note that since we have the product of two matrices equal to the identity, it suffices to check this identity if we multiply the matrices
in the reverse order.  But we have
\[
\begin{pmatrix}
E_1 & E_2 \\
D_1 & D_2
\end{pmatrix}
\begin{pmatrix}
C_1 & M_1 \\
C_2 & M_2
\end{pmatrix} \, =  \,
\begin{pmatrix}
E_1 & E_2 \\
0 & M_2^{-1}
\end{pmatrix}
\begin{pmatrix}
E_{1}^{-1} & M_1 \\
0       & M_2
\end{pmatrix}
 \, = \,
\begin{pmatrix}
I & EM \\
0 & I
\end{pmatrix}.
\]
But $EM = 0$ since the columns of $M$ are in the kernel of $E$.
\end{proof}

\begin{proof}[Conclusion of proof of Theorem \ref{thm:main}]
We must prove the implication $(3) \implies (2)$.
According to Lemma \ref{lem:matrixeq}, it suffices to
find integer matrices $C$ and $D$ which solve the matrix
equation $I + CE = MD$.  Lemma \ref{lem:zeqb} shows that
such integer matrices always exist for any strongly
connected graph $G$.
\end{proof}

Note that the proof of Theorem \ref{thm:main} tells us the precise 
form of an identifiable reparametrization that we can use
for any linear compartment model where the 
monomial cycles in the graph $G$ are identifiable.  In particular,
if this is the case, let $T$ be a spanning tree in 
the graph $G$, and set all the parameters associated to edges 
in that spanning tree equal to $1$.  The resulting model
has identifiable parameters associated to the remaining edges in the
graph.  Furthermore, and most importantly, that resulting model
can be obtained by a variable rescaling, thus it makes sense as a 
non-dimensionalization of the original model.  Note, however, 
that those inferred parameters are {\bf not} identifiable
parameters of the original model.  Although they are identifiable
in the model with some parameters set to $1$, they do not tell
us precise values in the original model, only information about
the relative changes in the parameters as we rescale the model.

\begin{ex} 
In Example \ref{ex:mainreparam}, we found an identifiable scaling reparametrization of Example \ref{ex:mainex}.  We now show how we attained this reparametrization, using Lemma \ref{lem:zeqb}.  Let a spanning tree $T$ correspond to the edges $a_{12},a_{23},a_{34}$ and use the monomial cycles described in Example \ref{ex:mainident}.  Then, setting the first column of $C$ to zero and solving $I + CE = MD$ using the solution from Lemma \ref{lem:zeqb}, we get that,
$$
C_1=
\begin{pmatrix}
1 & 1 & 1 \\
0 & 1 & 1 \\
0 & 0 & 1
\end{pmatrix},
D_2=
\begin{pmatrix}
1 & 0 & 0 \\
0 & 1 & 0 \\
0 & 0 & 1
\end{pmatrix}
$$
which corresponds to the scaling reparametrization $X_1=x_1$, $X_2=a_{12}x_2$, 
$X_3=a_{12}a_{23}x_3, X_4=a_{12}a_{23}a_{34}x_4$.
\end{ex}


\section{Dimension of the image of the double characteristic polynomial map } \label{sec:dim}

Theorem \ref{thm:main} reduces the problem of
deciding whether or not an identifiable scaling reparametrization
exists to calculating the dimension of the image of the
double characteristic polynomial map.  In this section and
the next, we derive results on this dimension proving 
some necessary and some sufficient conditions on graphs
that guarantee that the image of the double characteristic polynomial
map has the correct dimension.  
We also discuss the results of systematic computations for graphs
with small numbers of vertices.
To save ink, we introduce the following
definitions:

\begin{defn}
We say a graph $G$ with $n$ vertices and $m$ edges  has the \emph{expected dimension} if the image of the double characteristic polynomial map has dimension $m+1$.  The graph is \emph{maximal} if $m = 2n-2$.
\end{defn}

Clearly, a graph with more than $m = 2n-2$ edges cannot have the
expected dimension, since the double characteristic polynomial map
has image contained in $\rr^{2n-1}$.   Also, as indicated previously,
we need only consider graphs that are strongly connected and
we stick with that case throughout.

\begin{defn} Let $G$ be a directed graph.  We say that $G$ has an \emph{exchange} if there is a vertex $i$ such that $1 \to i$ and $i \to 1$ are
both edges in the graph.
\end{defn}

\begin{prop} \label{prop:exness} Suppose that $G$ is a strongly connected
maximal graph with the expected dimension.  Then $G$ has an exchange.
\end{prop}

\begin{proof} Let $A$ be the full $n$ by $n$ matrix in our ODE system and let $A_1$ be the $n-1$ by $n-1$ matrix where the first row and first column have been deleted.  Assume there is no exchange with compartment 1.  Then this means any 2 by 2 principal minor of $A$ involving the $(1,1)$ position will be of the form $a_{11}a_{ii}$ for $i=2,\ldots,n$ since no exchange with compartment 1 means that either $a_{1i}$ or $a_{i1}$ is zero.  Note that $c_1(A)$ corresponds to the (negated) trace of $A$, $c_2(A)$ corresponds to the sum of all principal 2 by 2 minors of $A$, $d_1(A_1)$ corresponds to the (negated) trace of $A_1$ and $d_2(A_1)$ corresponds to the sum of all principal 2 by 2 minors of $A_1$.  Then we have the relationship $c_2(A)=(c_1(A)-d_1(A_1))d_1(A_1)+d_2(A_1)$.  Thus the coefficients of the input-output equation are algebraically dependent.
\end{proof}

On the other hand, an exchange is not necessary for a graph
to have the expected dimension if the graph is not maximal.

\begin{prop}
Let $G$ be a strongly connected graph with $n$ vertices and $n$ edges
(that is, $G$ is a directed cycle).  Then $G$ has the expected dimension.
\end{prop}

\begin{proof}
The graph $G$ contains only one cycle $K$, which passes through all the vertices.
This means that the characteristic polynomial of $A_1$ is 
$$
(\lambda - a_{22})(\lambda - a_{33})  \cdots (\lambda - a_{nn}).
$$					
Since the roots of a polynomial can be determined from its coefficients, then 
all of $a_{22}, \ldots, a_{nn}$ are locally identifiable.
Parameter $a_{11}$ is identifiable (in fact, for any graph) by the formula
$a_{11}  = -c_1 + d_1$.  
Since $c_n  =  a_{11} \cdots a_{nn}  + (-1)^{n-1} K$ and $d_{n-1}  =  a_{22} \cdots a_{nn}$, we have 
$K =  (-1)^{n-1} (c_n +  (c_1 - d_1) d_{n-1})$ so the cycle $K$ is 
also identifiable.
\end{proof}

Next we consider situations where we can perform modifications
to the graph $G$ and preserve the property that $G$ has the
expected dimension.

\begin{prop} \label{prop:addtwo} 
Let $G$ be a graph that has the expected dimension.  Let $G'$ be
a new graph obtained from $G$ by adding a new vertex $1'$ and an exchange
$1 \to 1'$, $1' \to 1$, and making $1'$ be the new input-output node.
Then $G'$ has the expected dimension as well.
\end{prop}

\begin{proof} Let $A$ be the full matrix associated to the graph $G'$, $A_1$ be the matrix where the first row and first column have been deleted (and, hence
associated to the graph $G$), and $A_2$ be the matrix where the first two rows and first two columns have been deleted.  
We assume that the dimension of the image of the double characteristic polynomial
map associated to $G$ is $m+1$, and we want to show that 
for $G'$ we get $m+3$.

Let the characteristic polynomials $ \det(\lambda{I}-A)$, 
$ \det(\lambda{I}-A_1)$, and $ 
\det(\lambda{I}-A_2)$  be written (respectively) as:
$$\lambda^{n}+C_1\lambda^{n-1}+ \cdots + C_{n-1}\lambda+C_{n}$$ $$\lambda^{n-1}+c_1\lambda^{n-2}+\cdots+c_{n-2}\lambda+c_{n-1}$$
$$\lambda^{n-2}+d_1\lambda^{n-3}+\cdots+d_{n-3}\lambda+d_{n-2}$$
Then $\det(\lambda{I}-A)$ can be expanded as:
\begin{equation}\label{eq:expand}
\det(\lambda{I}-A)=(\lambda-a_{11})\det(\lambda{I}-A_1)-a_{12}a_{21}\det(\lambda{I}-A_2).
\end{equation}
This means $\det(\lambda{I}-A)$ can be written as: $\lambda^n+(-a_{11}+c_1)\lambda^{n-1}+(-a_{11}c_1+c_2-a_{12}a_{21})\lambda^{n-2}+(-a_{11}c_2+c_3-a_{12}a_{21}d_1)\lambda^{n-3}+\cdots+(-a_{11}c_{n-2}+c_{n-1}-a_{12}a_{21}d_{n-3})\lambda-a_{11}c_{n-1}-a_{12}a_{21}d_{n-2}$.

The double characteristic polynomial map associated to the graph $G'$
involves the characteristic polynomials of $A$ and $A_{1}$.  So
looking at the first two nontrivial coefficients of $ \det(\lambda{I}-A)$,
which are $-a_{11}+c_1$  and $-a_{11}c_1+c_2-a_{12}a_{21}$,
we can use the coefficients of $ \det(\lambda{I}-A_1)$ to solve for
$a_{11}$ and the cycle $a_{12}a_{21}$.  Hence, both of those coefficients
are identifiable functions.
Then Equation (\ref{eq:expand}) allows us to solve for the coefficients
of $\det(\lambda{I}-A_2)$.  Then, since we can perform rational manipulations
to solve for $a_{11}$, $a_{12}a_{21}$, and the coefficients of the
characteristic polynomials $ \det(\lambda{I}-A_1)$ and $\det(\lambda{I}-A_2)$,
this implies that the dimension of the image of the
 double characteristic polynomial
map associated to $G'$ is $m+3$ as desired.
\end{proof}

For the remainder of this section we prove a constructive
result which allows us to take a model with the expected
dimension and produce a new model with the expected
dimension adding one new vertex.  This construction
depends on the graph having a chain of cycles.

\begin{defn}
A \emph{chain of cycles} is a graph $H$ which consists of a sequence
of directed cycles that are attached to each other in a 
chain, by joining at the vertices.
\end{defn}

\begin{rmk} The graph in Example \ref{ex:mainex} contains a chain of cycles as a
subgraph,  where $a_{12}a_{21}$ and $a_{23}a_{34}a_{42}$ are the directed cycles that are attached to each other in a chain.  Figure 2 shows a general chain of three cycles.
\end{rmk}

\begin{figure}[h]
\begin{center}
\resizebox{!}{2cm}{
\includegraphics{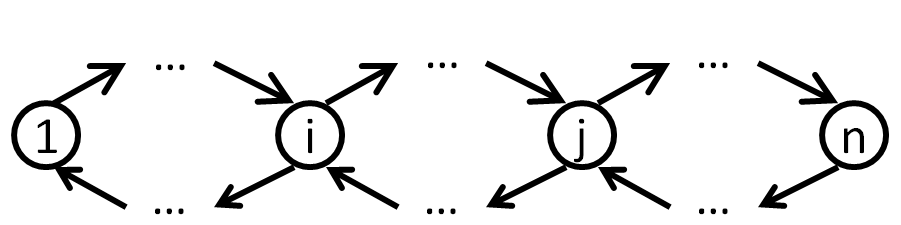}}
\end{center}\caption{A chain of cycles}
\end{figure}

\begin{thm}\label{thm:chainofcycles}
Let $G'$ be a graph that has the expected dimension with $n-1$
vertices.  Let $G$ be a new graph obtained from $G'$ by adding
a new vertex $n$ and two edges $k \to n$ and $n \to l$
and such that $G$ has a chain of cycles containing both
$1$ and $n$.  Then $G$ has the expected dimension.
\end{thm}

To prove Theorem \ref{thm:chainofcycles} requires a number of key ideas
which are assembled together in the present section.  
One key tool in the argument is to use a degeneration strategy, via
Gr\"obner bases.

Consider a $\kk$-algebra 
homomorphism $\phi^{*}: \kk[x] =  \kk[x_{1}, \ldots, x_{n}] 
\rightarrow \kk[y] =  \kk[y_{1}, \ldots, y_{m}]$.  Let $\omega \in \qq^{m}$
be a \emph{weight vector} on the polynomial ring $\kk[y]$. 
This induces a weight order on the polynomial ring $\kk[y]$ by 
which we can extract initial forms.  The weight of a monomial
$y^{a}$ is defined to be $\omega \cdot a$, and for a polynomial
$f$, the initial form ${\rm in}_{\omega}(f)$ is the sum
of all terms of $f$ whose monomial has the highest weight.

Since $\phi^{*}: \kk[x] \rightarrow \kk[y]$ is a $\kk$-algebra
homomorphism, it is described by the image polynomials $\phi(x_{i}) = f_{i}$.
Define the initial homomorphism $\phi^{*}_{\omega}: \kk[x]
 \rightarrow \kk[y]$ by $\phi^{*}_{\omega}(x_{i}) = {\rm in}_{\omega}(f_{i})$,
obtained by taking initial terms of all the polynomials $f_{i}$.

The map $\phi^{*}$ and the weight vector $\omega$ also induce a 
weight order on $\kk[x]$.  The induced weight of $\phi^{*} \omega$
is defined so that the weight of $x_{i}$ is equal to the largest
$\omega$ weight of any monomial appearing in $f_{i}$.

\begin{lemma}\label{lem:initial}
Let $\omega \in \qq^{m}$ be a weight vector and $\phi^{*}: \kk[x] \rightarrow
\kk[y]$ be a $\kk$-algebra homomorphism, let $I = \ker \phi^{*}$ and $I' = \ker \phi^{*}_{\omega}$.  Then
$$
{\rm in}_{\phi^{*}\omega} I  \subseteq  I'.
$$
\end{lemma}

This is a standard result in the theory of SAGBI bases, see e.g.~\cite[Lemma 11.3]{Sturmfels1996}.  Note that for a 
polynomial parametrization $\phi : \kk^{m} \rightarrow \kk^{n}$, 
$\phi^{*}:\kk[x] \rightarrow \kk[y]$, 
denotes the pullback map, i.e.~the corresponding $\kk$-algebra homomorphism.
Hence, we can define the initial parametrization $\phi_{\omega}$ to be
the parametrization with pullback $\phi^{*}_{\omega}$.

\begin{cor}\label{cor:initialdim}
Let $\phi^{*} :\kk[x] \rightarrow \kk[y]$ be a $\kk$-algebra homomorphism and
$\omega \in \qq^{m}$ a weight vector.  Then 
$$
\dim ({\rm image} \,   \phi_{\omega})  \leq  \dim ({\rm image} \, \phi).
$$
\end{cor}

\begin{proof} 
The dimension of the image of a polynomial parametrization $\phi$ is
equal to the Krull dimension of the quotient ring  $\kk[x]/\ker \phi^{*}$.
We can speak of the dimension of an ideal, rather than the dimension
of a ring.  For any weight vector, we always have
$\dim I  =  \dim  {\rm in}_{\omega} I$.  And if $I \subseteq J$, then
$\dim J \leq \dim I$.
Thus, using the ideals in Lemma \ref{lem:initial} we have
$$ \dim I'  \leq \dim {\rm in}_{\phi^{*}\omega} I  = \dim I,$$
which completes the proof.
\end{proof}

Here is how we will use Corollary \ref{cor:initialdim}.  We want to compute the dimension of the image of a polynomial
parametrization $\phi$.  We know for other reasons an upper bound $d$ on this
dimension.  We have a weight vector $\omega$ where we 
can compute the dimension of the image
of the polynomial parametrization $\phi_{\omega}$, and we show
it is equal to $d$.
Then, by Corollary \ref{cor:initialdim}, we know that the
dimension of the image of $\phi$ must be $d$.
At a key step we compute the Jacobian of the transformation
to calculate the dimension of the image of the double characteristic
polynomial map.

\begin{proof}[Proof of Theorem \ref{thm:chainofcycles}]
Let
$\phi_{G} : \rr^{n +m} \rightarrow \rr^{2n-1}$ be
the double characteristic polynomial map associated to the
graph $G$. 
The $\kk$-algebra homomorphism of interest is
$\phi^{*}_{G}: \kk[c,d]  \rightarrow \kk[a]$
where $c, d$ are the appropriate characteristic polynomial coefficients.
Choose a weight vector $\omega$, a weighting on $\kk[a]$ such that
$$
\omega_{ij}  = \left\{ \begin{array}{cl}
 0  &  \mbox{ if } (i,j) = (n,n)  \\
 \frac{1}{2} & \mbox{ if } (i,j) = (k,n), (n,l) \\
 1 &  \mbox{ otherwise}.
\end{array}\right.
$$
Since all the polynomial functions in $\phi_{G}$ that appear
are homogeneous, this has the effect of removing any term
that involves a cycle incident to the vertex $n$, except
for the constant coefficients of the characteristic
polynomials. In this case, every term involves a cycle incident to
$n$, and all such terms will have weight $n-1$
for the full characteristic polynomial of $A$, and weight $n-2$
for the characteristic polynomial of $A_{1}$.
In other words, with the specific choice of
weighting $\omega$ above, we have:
\begin{eqnarray*}
\phi^{*}_{G,\omega}(c_{i}) & = & \phi^{*}_{G'}(c_{i})  \quad 
i = 1, \ldots, n-1  \\
\phi^{*}_{G,\omega}(d_{i}) & = & \phi^{*}_{G'}(d_{i})  \quad
 i = 1, \ldots, n-2  \\
\phi^{*}_{G,\omega}(c_{n}) & = & \phi^{*}_{G}(c_{n})   \\
\phi^{*}_{G,\omega}(d_{n-1}) & = & \phi^{*}_{G}(d_{n-1})
\end{eqnarray*}

In other words, the parametrization $\phi_{G,\omega}$ agrees
with  $\phi_{G'}$ except in its two new coordinates, where
it matches $\phi_{G}$.
Our goal now is to prove that the image of this
parametrization $\phi_{G,\omega}$ has dimension $2$ more than
the dimension of the image of $\phi_{G'}$, since this is the
largest increase in dimension that is possible.

For a map $\phi$, let $J(\phi)$ denote the Jacobian matrix.
The rank of the Jacobian matrix at a generic point gives
the dimension of the image of the map $\phi$.   
Note that generic means ``except possibly for a proper subvariety of
the parameter space''.

In our case, the Jacobian of $\phi_{G,\omega}$ 
is a $(2n-1) \times (n+m)$ matrix, whose
columns correspond to the $c$'s and $d$'s and whose rows
are labeled by the nonzero entries of $A$.  Sort the
rows and columns so that the last two rows are labeled
by $c_{n}$ and $d_{n-1}$, and the last three columns
are labelled by $a_{nn}$, $a_{kn}$ and $a_{nl}$.
With this convention on the orders of rows and columns
of the Jacobian matrix $J(\phi_{G,\omega})$, it is a block
matrix of the form
$$
J(\phi_{G,\omega})  =  \begin{pmatrix}
J(\phi_{G'})  &  0  \\
*   &  C
\end{pmatrix}
$$
where $J(\phi_{G'})$ is the $(2n-3) \times (n+m-2)$ Jacobian
matrix of $\phi_{G'}$, and $C$ is the $2 \times 3 $ matrix
\begin{equation}\label{eq:C}
C = 
\begin{pmatrix}
\frac{\partial c_{n}}{\partial a_{nn}} & 
\frac{\partial c_{n}}{\partial a_{kn}} &
\frac{\partial c_{n}}{\partial a_{nl}}  \\ 
\frac{\partial d_{n-1}}{\partial a_{nn}} & 
\frac{\partial d_{n-1}}{\partial a_{kn}} &
\frac{\partial d_{n-1}}{\partial a_{nl}}
\end{pmatrix}.
\end{equation}

By assumption the rank of $J(\phi_{G'})$ is generically
equal to $m -1$.   Since $J(\phi_{G,\omega})$ is a block triangular
matrix, it suffices to show that the matrix $C$ generically
has rank $2$.  Furthermore, we can show this by
exhibiting a single choice of the parameters $A$ that
yields a matrix $C$ with rank $2$, since having full rank
is a Zariski open condition on the parameters.
We work now on finding a matrix $A$ which gives the rank of $C$ equal to 2.

In particular, let $H$ be a chain of cycles in $G$ that
contains both $1$ and $n$.  We can assume that $1$ and $n$
are at the two opposite ends of the chain.
Suppose that the cycles in $H$ are $s_{1}, \ldots, s_{t}$
in order, so that $1$ is in cycle $s_{1}$ and $n$ is
in cycle $s_{t}$.

Choose the matrix $A$ by setting all diagonal entries to $1$,
$a_{ij} = 0$ for all edges $i \to j \not\in H$.  For all the edges
in $H$, for each cycle $s_{i}$, choose the edge weights so that
the product of edges' weights is equal to $(-1)^{\ell(s_{i}) -1}$.
For the cycle that contains the vertex $n$, we further require
that both $a_{kn}$ and $a_{nl}$ (the unique incoming and outgoing
edges to $n$) are set to $1$.

With these choices for the matrix $A$, each of the entries in the
matrix $C$ will be a nonnegative integer, equal to the number
of monomials in that polynomial entry involving only edges
from the cycles $s_{1}, \ldots, s_{t}$, together with the
trivial cycles at each node.  We must count the number of ways
to do this in each of the cases.

We handle two cases.  First when $t \geq 2$.

First consider the entry $\frac{\partial c_{n}}{\partial a_{nn}}$.
The only nonzero monomials appearing here will arise from
taking appropriate products of the cycles
$s_{1}, \ldots, s_{t-1}$, since the cycle $s_{t}$ cannot be involved.
Since each cycle touches its two neighboring cycles, and no
other cycles, and in the expansion we expand over all products of 
nontouching cycles that cover all $n$ vertices, we see that
the number of monomials will equal the number of subsets of 
$\{1, \ldots, t-1\}$, with no adjacent elements.  By Lemma \ref{lem:fibonacci}
this is the Fibonacci number $F_{t+1}$.

When we consider the entry $\frac{\partial d_{n-1}}{\partial a_{nn}}$,
the only nonzero monomial appearing here will arise from taking products
of the cycles $s_{2}, \ldots, s_{t-1}$ since
neither of the cycles $s_{1}$ nor $s_{t}$ can be involved.  By
a similar argument as the preceding paragraph we see that
this will give the Fibonacci number $F_{t}$.

Now when we consider the entry $\frac{\partial c_{n}}{\partial a_{kn}}$
or equivalently $\frac{\partial c_{n}}{\partial a_{nl}}$
we must use the cycle $s_{t}$.  This prohibits us from using 
the cycle $s_{t-1}$.  Hence, we are counting appropriate
products of the cycles $s_{1}, \ldots, s_{t-2}$.
This will give us the Fibonacci number $F_{t}$.

Finally with the entry $\frac{\partial d_{n-1}}{\partial a_{kn}}$
or equivalently $\frac{\partial d_{n-1}}{\partial a_{nl}}$
we must use the cycle $s_{t}$ and thus we cannot use the cycles
$s_{1}, s_{t-1}$.  Hence we are counting appropriate products of
the cycles $s_{2}, \ldots, s_{t-2}$.  This will give the Fibonacci number
$F_{t-1}$.

Hence, the submatrix $C$ of the Jacobian matrix has the following form
for this choice of parameters:
$$
C  = \begin{pmatrix}
F_{t+1} & F_{t} & F_{t} \\
F_{t} & F_{t-1} & F_{t-1}
\end{pmatrix}.
$$
The classical identity of Fibonacci numbers $F_{t+1}F_{t-1} - F_{t}^{2} = (-1)^{t}$ guarantees that this matrix has full rank.

In the case where $t = 1$, the same argumentation works until the
analysis of $\frac{\partial d_{n-1}}{\partial a_{kn}}$.
Since $1$ is involved in the cycle $s_{1}$, there will be no
monomials, and thus the polynomial $d_{n-1}$ is identically zero.
Since $F_{0} = 0$, then the
matrix $C$ has the same shape as above, and we still deduce that
$C$ has rank $2$.
\end{proof}

\begin{lemma}\label{lem:fibonacci}
The number of subsets $S$ of $\{1,2, \ldots, n\}$ 
such that $S$ contains no pair of adjacent numbers is
the $n+2$-nd Fibonacci number, $F_{n+2}$
which satisfies the recurrence 
$F_{0} = 0, F_{1} =1, F_{n+1} = F_{n} + F_{n-1}$.
\end{lemma}

We can apply Theorem \ref{thm:chainofcycles} to analyze
inductively strongly connected graphs.

\begin{defn}
A directed graph $G$ is \emph{inductively strongly connected}
if each of the induced subgraphs $G_{\{1, \ldots, i\}}$
is strongly connected for $i = 1, \ldots, n$ for some ordering of the 
vertices $1,\ldots,i$ which must start at vertex $1$.
\end{defn}

\begin{prop}\label{prop:inductedges}
If $G$ is inductively strongly connected with $n$ vertices, then $G$ has
at least $2n -2$ edges.
\end{prop}

\begin{proof}
By induction, if a graph with $n-1$ vertices is inductively strongly connected
it has at least $2n -4$ edges.  Adding the $n$th vertex requires adding at least two edges, one into $n$ and one out of $n$, to get
a strongly connected graph.
\end{proof}

The proof of Proposition \ref{prop:inductedges} shows that every 
inductively
strongly connected graph contains a subgraph of exactly $2n - 2$
edges, obtained by adding only one in and one out edge of vertex
$i$ at step $i$ in the construction.  An inductively
strongly connected graph with exactly $2n-2$ edges is a 
\emph{minimal inductively strongly connected graph}.

\begin{thm}\label{thm:inductive}
Let $G$ be a minimal inductively strongly connected graph with $n$ vertices.
Then the dimension of the image of the double characteristic polynomial
map is $2n-1$.
\end{thm}

\begin{proof}
By Theorem \ref{thm:chainofcycles} and the inductive nature
of inductively strongly connected graphs, it suffices to show
that every inductively strongly connected graph has a chain
of cycles containing the vertices $1$ and $n$.

We prove this by induction on $n$.  Since $G$ is inductively strongly connected
there is a nontrivial  cycle $c$ that passes through the vertex $n$.
If $c$ contains $1$, we are done.  Otherwise, let $i$ be the smallest
vertex appearing in $c$, and let $G'$ be the induced subgraph on
$\{1,2, \ldots, i\}$.  By induction, $G'$ has a chain of cycles $H$
containing $1$ and $i$.  Attaching $c$ to $H$ gives a chain of cycles
in $G$ containing $1$ and $n$.
\end{proof}


\section{Algorithms and Computations}\label{sec:comp}

We now summarize our results from the previous sections 
and present our work as an algorithm for testing the 
existence of and finding identifiable scaling reparametrizations
for a specific family of linear ODE models.  

\begin{alg}\label{alg:mainalg} (Computing an identifiable scaling reparametrization)

\noindent Input: A strongly connected graph $G$ with $n$ vertices and $m\leq{2n-2}$ edges. \\
Output: Either an identifiable scaling reparametrization or a statement that one does not exist. 
\begin{enumerate}
\item Compute $d=$ dimension of the image of the double characteristic polynomial map $c$. 
\item If $d\neq{m+1}$, then an identifiable scaling reparametrization does not exist. Otherwise:
\begin{enumerate}
\item Find a spanning tree $T$ of $G$, with edges $j_{1} \to i_{1}$, $\ldots$, $j_{n-1} \to i_{n-1}$. 
\item Form the matrix $E$ by rearranging the columns of $E(G)$ so that the first $n-1$ columns correspond to edges in $T$ and by deleting the first row.  In other words, $E=(E_1 \ E_2)$, where $E_1$ is an $n-1$ by $n-1$ matrix corresponding to the edges in $T$.  
\item Determine the monomial scaling $X_i=f_i(A)x_i$.
Set $f_{1}(A) = 1$.  Let $r_{i} = (r_{1,i}, \ldots, r_{n-1,i})^{T}$
be the $i$th column of $C_{1} =E_1^{-1}$.  Then
$f_{i+1}(A) = a_{i_{1}j_{1}}^{r_{1,i}} \cdots a_{i_{n-1}j_{n-1}}^{r_{n-1,i}}$.
\item Replace the entries $a_{ij}$ of $A$ with the new entries 
$a_{ij}f_i(A)/f_j(A)$.
\end{enumerate}
\end{enumerate}
\end{alg}

In Step 1, $d$ can be computed by either calculating the rank of the 
Jacobian matrix of $c$ at a generic point or by finding the vanishing ideal 
of the image of $c$ using Gr\"obner bases.  Step 1 can be sped up by 
first checking if $G$ is an inductively strongly connected graph. 
If so, the condition $d=m+1$ is automatically satisfied in Step 2
and thus $d$ need not be computed using more time-consuming methods.  

If an identifiable scaling reparametrization exists, the new matrix 
$A$ will have the $n-1$ entries $a_{ij}$ corresponding to the spanning 
tree $T$ equal to $1$ and the remaining $m-n+1$ off-diagonal entries 
can be thought of as the new parameters in the reparametrized system. 
As noted in Section \ref{sec:mono}, these parameters are identifiable
in the reparametrized model, but are {\bf not} identifiable parameters 
of the original model.  The new $m-n+1$ parameters can be written 
in terms of the cycles of the graph $G$ using the 
following algorithm:

\begin{alg}\label{alg:optalg} (Writing new coefficients in terms of cycles)

\noindent Input: A strongly connected graph $G$ that 
has an identifiable scaling reparametrization and a spanning tree $T$, 
as determined by Algorithm \ref{alg:mainalg}. \\
Output: An identifiable scaling reparametrization in terms of 
cycles of the graph $G$.
\begin{enumerate}
\item Choose a set of $m-n+1$ linearly independent cycles 
of the graph $G$, $q_{1}, \ldots, q_{m-n+1}$.  
The $m-n+1$ off-diagonal entries $a_{ij}$ 
which correspond to edges not in $T$ can written as functions 
of these cycles using the following procedure:
\begin{enumerate}
\item Form the matrix $M$ whose columns are the 
exponent vectors of $q_{1}, \ldots, q_{m-n+1}$. Rearrange the rows of
 $M$ so that the first $n-1$ rows 
correspond to the edges in $T$.  In other words, $M$ is partitioned
 into $(M_1 \ M_2)^T$, where $M_1$ corresponds to $T$.  
\item 
Let $r_{i} = (r_{1,i}, \ldots, r_{m-n+1,i})^{T}$
be the $i$th column of the matrix $D_2=M_2^{-1}$, corresponding to the
edge $k \to j$.  Then the rescaling gives the $a_{jk}$ entry
as $q_{1}^{r_{1,i}} \cdots q_{m-n+1}^{r_{m-n+1,i}}$.
\end{enumerate}
\end{enumerate}
\end{alg}

We now demonstrate our algorithms on two additional examples.

\begin{ex}
Input: The graph $G$ in Figure 3, with the associated linear ODE system:

$$
\begin{pmatrix} 
\dot{x}_1 \\
\dot{x}_2 \\
\dot{x}_3 \\ 
\dot{x}_4 \end{pmatrix} = {\begin{pmatrix} 
a_{11} & a_{12} & 0 & 0 \\
a_{21} & a_{22} & a_{23} & 0 \\
0 & 0 & a_{33} & a_{34} \\
0 & a_{42} & a_{43} & a_{44} 
\end{pmatrix}} {\begin{pmatrix}
x_1 \\
x_2 \\
x_3 \\
x_4 \end{pmatrix} } + {\begin{pmatrix}
u_1 \\
0 \\
0 \\
0 \end{pmatrix}}
$$
$$ y=x_1.$$

\begin{figure}[h]
\begin{center}
\resizebox{!}{2cm}{
\includegraphics{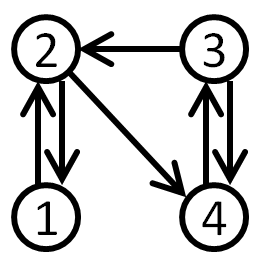}}
\end{center}\caption{A graph with four vertices}
\end{figure}

Output: No identifiable scaling reparametrization exists since $d=6$ does not equal $m+1=7$.
\end{ex}

\begin{ex}
Input: The graph $G$ in Figure 4, with the associated linear ODE system:

$$
\begin{pmatrix} 
\dot{x}_1 \\
\dot{x}_2 \\
\dot{x}_3 \\ 
\dot{x}_4 \\
\dot{x}_5 \end{pmatrix} = {\begin{pmatrix} 
a_{11} & 0 & a_{13} & 0 & a_{15} \\
a_{21} & a_{22} & 0 & 0 & 0 \\
a_{31} & a_{32} & a_{33} & a_{34} & 0 \\
0 & 0 & a_{43} & a_{44} & 0 \\
0 & 0 & 0 & a_{54} & a_{55}
\end{pmatrix}} {\begin{pmatrix}
x_1 \\
x_2 \\
x_3 \\
x_4 \\
x_5 \end{pmatrix} } + {\begin{pmatrix}
u_1 \\
0 \\
0 \\
0 \\
0 \end{pmatrix}}
$$
$$ y=x_1.$$

\begin{figure}[h]
\begin{center}
\resizebox{!}{2cm}{
\includegraphics{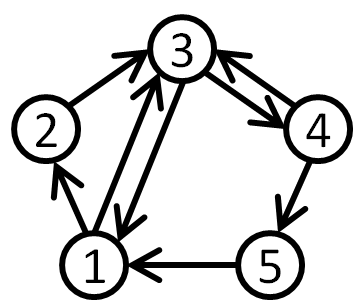}}
\end{center}\caption{A graph with five vertices}
\end{figure}

Output: The following identifiable scaling reparametrization:

$$
\begin{pmatrix} 
\dot{X}_1 \\
\dot{X}_2 \\
\dot{X}_3 \\ 
\dot{X}_4 \\
\dot{X}_5 \end{pmatrix} = {\begin{pmatrix} 
a_{11} & 0 & \frac{a_{13}a_{31}}{a_{43}a_{31}a_{15}a_{54}} & 0 & 1 \\
\frac{(a_{43}a_{31}a_{15}a_{54})(a_{13}a_{32}a_{21})} {a_{13}a_{31}} & a_{22} & 0 & 0 & 0 \\
a_{43}a_{31}a_{15}a_{54} & 1 & a_{33} & a_{34}a_{43} & 0 \\
0 & 0 & 1 & a_{44} & 0 \\
0 & 0 & 0 & 1 & a_{55}
\end{pmatrix}} {\begin{pmatrix}
X_1 \\
X_2 \\
X_3 \\
X_4 \\
X_5 \end{pmatrix} } + {\begin{pmatrix}
u_1 \\
0 \\
0 \\
0 \\
0 \end{pmatrix}}
$$
$$ y=X_1,$$

\noindent where $T$ corresponds to the edges $a_{32},a_{43},a_{54},a_{15}$,
 the rescaling is $X_1=x_1, X_2=a_{32}a_{43}a_{54}a_{15}x_2, 
 X_3=a_{43}a_{54}a_{15}x_3, X_4=a_{54}a_{15}x_4, X_5=a_{15}x_5$, 
 and the monomial cycles are $a_{13}a_{31}$, $a_{34}a_{43}$, 
 $a_{43}a_{31}a_{15}a_{54}$, and $a_{13}a_{32}a_{21}$.

Thus, the new reparametrized model has $m+1$ algebraically independent parameters 
$b_{ij}$ and can be written as:
$$
\begin{pmatrix} 
\dot{X}_1 \\
\dot{X}_2 \\
\dot{X}_3 \\ 
\dot{X}_4 \\
\dot{X}_5 \end{pmatrix} = {\begin{pmatrix} 
b_{11} & 0 & b_{13} & 0 & 1 \\
b_{21} & b_{22} & 0 & 0 & 0 \\
b_{31} & 1 & b_{33} & b_{34} & 0 \\
0 & 0 & 1 & b_{44} & 0 \\
0 & 0 & 0 & 1 & b_{55}
\end{pmatrix}} {\begin{pmatrix}
X_1 \\
X_2 \\
X_3 \\
X_4 \\
X_5 \end{pmatrix} } + {\begin{pmatrix}
u_1 \\
0 \\
0 \\
0 \\
0 \end{pmatrix}}
$$
$$ y=X_1.$$

The graph $G$ is inductively strongly connected and has $2n-2$ edges, and thus $d$ automatically equals $m+1=9$.
\end{ex}

We now describe results of our computations of small
graphs and some of the conjectures those computations
suggest.  In particular, we highlight graphs which 
do have the expected dimension but this cannot be deduced from
applying any of our constructions from Section \ref{sec:dim}.
At present we lack a conjecture which would claim to give
a complete characterization of all graphs which do have the
expected dimension, but we provide conjectures on the structure
in some extremal cases.

Below is a table displaying the results of our computations
for all relevant graphs up to $n = 5$ vertices.  These computations were performed in Mathematica \cite{Wolfram}.  { We compute the rank of the Jacobian of the double characteristic polynomial map at two randomly sampled points in parameter space to determine if the
graph $G$ has the expected dimension.  

Here we partition the graphs by the number $n$ of vertices and
the number $m$ of edges with $n \leq m \leq 2n-2$.  The columns of
the table record the following information:
\begin{enumerate}
\item[A:] The number of strongly connected graphs with $n$ vertices and $m$ edges.
\item[B:] The number of graphs from A that have the expected dimension. 
\item[C:] The number of strongly connected graphs up to symmetry permuting vertices $2, \ldots, n$. 
\item[D:] For the maximal case, $m=2n-2$, the number of strongly connected graphs up to symmetry with an exchange. 
\item[E:] The number of graphs from C that have the expected dimension.
\item[F:] For the maximal case, $m=2n-2$, the number of inductively strongly connected graphs up to symmetry. 
\end{enumerate}

\begin{center}
    \begin{tabular}{ | p{2cm} | p{1cm} | p{1cm} | p{1cm} | p{1cm} | p{1cm} | p{1cm} |}
    \hline
   $(n,m)$ & A  & B  & C & D  & E   & F \\ \hline
    (3,3)  & 2 & 2 & 1 & NA & 1    & NA  \\ \hline
    (3,4)  & 9 & 7 & 5 & 4 & 4 & 4  \\ \hline
    (4,4)  & 6 & 6 & 1 & NA & 1  & NA \\ \hline
    (4,5)  & 84 & 54 & 15 & NA &  12  & NA  \\ \hline
    (4,6)  & 316 & 166 & 55 & 34 &  30  & 26  \\ \hline
    (5,5)  & 24 & 24 & 1 & NA & 1  & NA \\ \hline
    (5,6)  & 720 & 576 & 32 & NA & 26  & NA \\ \hline
    (5,7)  & 6440 & 4052 & 281 & NA & 180   & NA \\ \hline
    (5,8)  & 26875 & 9565 & 1158 & 581 & 421  & 267  \\ \hline
    \end{tabular}
\end{center}

\begin{rmk} From the table we see that, for the maximal case when $m=2n-2$, not every graph with an identifiable reparametrization is inductively strongly connected.  Figure 5 displays the four graphs up to symmetry that have an identifiable reparametrization but are not inductively strongly connected, for $n=4$ and $m=6$.

\begin{figure}[h]
\begin{center}
\resizebox{!}{2cm}{
\includegraphics{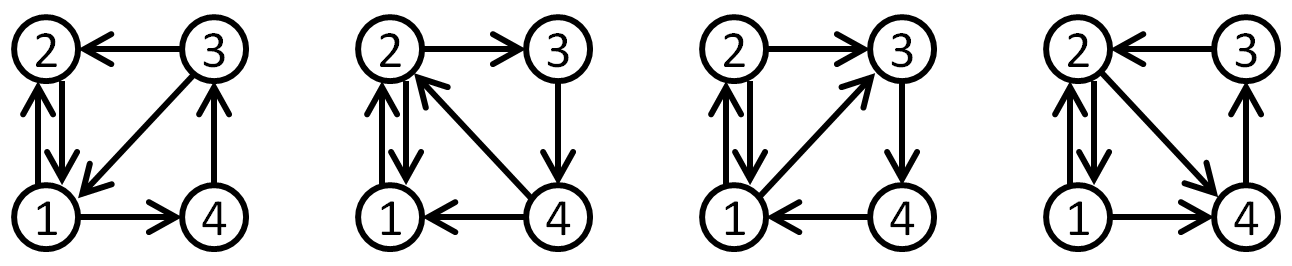}}
\end{center}\caption{Graphs with an identifiable reparametrization but not inductively strongly connected}
\end{figure}
\end{rmk}

\begin{defn} Let $G$ be a directed graph with $n +1$ vertices labelled
$0,1,2, \ldots, n$, where $0$ is the distinguished vertex corresponding
to the input-output compartment.  Suppose that $G$ has an exchange
with vertex $1$.  The \emph{collapsed graph} $G'$ is the new graph with
$n$ vertices $1, \ldots, n$, where vertices $0$ and $1$ have been
identified.  So an edge $i \to j$ appears in $G'$ if it appears in $G$ or
if $i = 1$ and $0 \to j$ appears in $G$.  The vertex $1$ in $G'$ is
the new distinguished vertex of the input-output compartment.
\end{defn}

Here are two conjectures about how having the expected dimension
is preserved under collapsing an exchange.

\begin{conj}\label{conj:collapse2}
Let $G$ be a graph with $n$ vertices and $2n-2$ edges with an exchange, and let $G'$ be the resulting collapsed
graph.  If $G'$ has $2n-4$ edges with an exchange, then $G$ has the expected dimension if and only if $G'$ has the 
expected dimension.
\end{conj}

\begin{conj}\label{conj:collapse3}
Let $G$ be a graph with $n$ vertices and $\leq 2n-2$ edges with an exchange, and let $G'$ be the resulting collapsed
graph.  If $G'$ has $n-1$ edges, then $G$ has the expected dimension if and only if $G'$ has the 
expected dimension.
\end{conj}

Some supporting evidence for these conjectures is provided by Proposition 
\ref{prop:addtwo}, where it is possible to collapse an exchange if those
are the only edges incident to vertex $1$.  Also, 
in the case where $G$ is an inductively strongly connected
graph, the collapsing preserves the property of being inductively strongly
connected, and hence Conjecture \ref{conj:collapse2} is true in that case.  

\begin{prop}
Let $G$ be an inductively strongly connected graph, and let $G'$ be
the graph obtained by collapsing the vertices in the first exchange.
Then $G'$ is inductively strongly connected.
\end{prop}

Note that since the induced subgraph $G_{1,2}$ is strongly connected,
every inductively strongly connected graph has an exchange that can
be collapsed.

\begin{proof}
We proceed by induction on the number of vertices.  Let $G$ have
$n$ vertices and be inductively strongly connected.  Let $\tilde{G} = G_{\{1, \ldots, n-1\}}$ be the induced subgraph on the first $n-1$ vertices.
This is inductively strongly connected.  Its collapsing $\tilde{G}'$
is inductively strongly connected by induction.  The graph
$G'$ is obtained from $\tilde{G}'$ by adding the vertex $n$ and at least
one incoming edge to and one outgoing edge from $n$, which makes
$\tilde{G}'$ inductively strongly connected.
\end{proof}


\section*{Acknowledgments}

We would like to thank Marisa Eisenberg and Hoon Hong for 
their constructive comments concerning this work.  
Nicolette Meshkat was partially supported by the David and Lucille Packard Foundation.
Seth Sullivant was partially supported by the David and Lucille Packard 
Foundation and the US National Science Foundation (DMS 0954865).

\end{document}